\documentclass{article}
\usepackage[utf8]{inputenc}
\usepackage{amsthm}

\theoremstyle{definition}
\newtheorem{defn}{Definition}
\newtheorem{thm}{Theorem}
\newtheorem{cor}{Corollary}
\usepackage{tikz}
\usepackage{geometry}[margin=3in]
\usepackage{caption}
\usepackage{hyperref}
\usepackage{subcaption}
\usepackage{tcolorbox}
 \usepackage[
    backend=biber,
    style=alphabetic,
    citestyle=phys,
  ]{biblatex}

\title{Equal Splits of Vertex-Weighted Trees}
\author{Corinne Mulvey
\\Mentored by Diana Davis}
\date{October 2020}

\begin{document}

\maketitle

\begin{abstract}
  Given a tree of weighted vertices, it is sometimes possible to break the tree into two equally-weighted subtrees within an allowable error. We give a fast algorithm that finds an edge which breaks the tree into equal-weight components or determines there is no such edge.   
\end{abstract}

\section{Introduction}
In recent years, the issue of political gerrymandering has attracted the interest of mathematicians. The leading algorithm for processing districting plans, GerryChain \cite{gerrychain}, uses randomly generated spanning trees of districting plans to construct millions of possible districting plans. These plans are then implemented using existing electoral data to determine if the currently enacted plan is an outlier.

To create new districts, the algorithm merges two existing districts, generates a spanning tree over the combined district, then searches for an edge of this tree that produces two subtrees of approximately equal population when removed. To find this edge, the algorithm randomly selects edges and tests if they yield equal population subtrees. After too many failures, the algorithm generates a new spanning tree and repeats the process \cite{deford2019recombination}. Although the algorithm terminates after too many failures, this does not necessarily mean there is not an edge that splits the tree into equal-weight components. Thus, making this process more efficient is of great interest. 

We give a fast algorithm to replace this process by efficiently determining if a given weighted tree contains a edge that breaks it into equally weighted subtrees within an allowable error. Furthermore, it can quickly determine if there is no such edge. The weight of each vertex corresponds to the population of the voting precinct the vertex represents. 
\subsection{Notation}
\begin{itemize}
    \item The edge set and vertex set of a tree $T$ is denoted as $E(T)$ and $V(T)$, respectively. 
    \item The degree of a vertex $v$, $deg(v)$, denotes the number of edges incident with $v$.  
    \item The forest $T-\{v\}$ is the set of disconnected subtrees formed by removing the vertex $v$ and all the edges incident with $v$ from a tree $T$.
    \item For a vertex $v$ in a weighted tree $T$, let $w(v)$ denote the weight of the vertex and $w(T)$ denote the total vertex weight of the tree. 
    
\end{itemize}

\section{Splitting Vertex-Weighted Trees}
As mentioned in the introduction, given a tree of weighted vertices, it is useful to know if there is a single edge that can be removed to produce two trees with approximately equal total vertex weights. This motivates the following definitions of a 2-splittable tree and
cut edge. 

\begin{defn}[2-splittable for weighted trees]
Let $T$ be a vertex weighted tree with total vertex weight $w(T)= S$. An edge $e$ is a \textbf{cut edge} of $T$ if the components $T_1$ and $T_2$ of the forest $T-\{e\}$ are such that,
$$\frac{S}{2}-\epsilon \leq w(T_1), w(T_2) \leq \frac{S}{2}+\epsilon,$$ for $\epsilon\geq0$.
If a tree $T$ contains a cut edge, then $T$ is \textbf{2-splittable within an error of $\epsilon$}.
\end{defn}
\begin{thm}
Let $T$ be a tree with total vertex weight $S$. Suppose that there exists a vertex $v$ in $T$ such that every component of the forest $T-\{v\}$ has total vertex weight less than $\frac{S}{2}-\epsilon$. Then $T$ is \emph{not} 2-splittable. 
\label{thm:beans}
\end{thm}

\begin{proof}
We prove the contrapositive. Suppose such a vertex $v$ exists and $T$ is 2-splittable. Since there is at least one cut edge $e$, the forest $T-\{e\}$ consists of components $K_1$ and $K_2$ with total vertex weights such that $\frac{S}{2}-\epsilon\leq w(K_1),w(K_2)\leq \frac{S}{2}+\epsilon$. 

Consider the two possible locations of $e$. Either $e$ is \emph{not} in the graph $T-\{v\}$ or $e$ \emph{is} in $T-\{v\}$. 

\emph{Case I.} (See Figure \ref{fig:pfa}.) If $e\not\in E(T-\{v\})$, then $e$ is incident with $v$ and an adjacent vertex $v'$ in a component $T_i$ of $T-\{v\}$. In the graph $T-\{e\}$, $v$ and $v'$ are in distinct components $K_1$ and $K_2$. Without loss of generality, let  $v'$ be in the component $K_1$. Then, $K_1$ is the set of all vertices and edges connected to $v'$ in $T-\{v\}$. Therefore, $T_i\cong K_1$, and 

$$w(T_i)=w(K_1)\geq \frac{S}{2}-\epsilon.$$

 When the cut edge is adjacent to $v$, the component incident with a vertex incident with $e$ must have total vertex weight within the desired range $\frac{S}{2}\pm\epsilon$. Thus, that component has total weight of at least $\frac{S}{2}-\epsilon$.

\emph{Case II.} (See Figure \ref{fig:pfb}.) If $e\not\in T-\{v\}$, then $e$ is in some component $T_i$ of $T-\{v\}$. Without loss of generality, assume that $v$ is in the component $K_2$ of the graph $T-\{e\}$.  Removing the vertex $v$  causes $K_2$ to be disconnected in the graph $T-\{v\}$ with some of its edges and vertices in the component $T_i$. The other component of $T-\{e\}$, $K_1$, remains connected in $T_i$ because $e$ is in the connected component $T_i$ of $T-\{v\}$.  Therefore, $T_i$  consists of the entire vertex set of $K_1$ as well as a subset of vertices of $K_2$, and we have,
$$w(T_i)>w(K_1)\geq \frac{S}{2}-\epsilon.$$

Thus, the total vertex weight of a component $T_i$ of $T-\{v\}$ is greater than $\frac{S}{2}-\epsilon$. 
\end{proof}
\begin{figure}[h!]
    \centering
    \begin{subfigure}[h]{.4\textwidth}
    \centering
    \begin{tikzpicture}
        \node[circle, draw=black!50](v) at (-.2,0){$v$};
        
        \node[circle, draw=black, inner sep=.7mm ](vi) at (1.2,0){$v'$}; 
        
        \node[circle, draw=black!50](v1) at (-1.2,0){};

        \node[circle, draw=black!30](v1a) at (-2,.5){};
        
        \node[circle, draw=black!30](v1b) at (-2,-.5){};
        
        \node[circle, draw=black](vi1) at (2.2,0){};
        
        \node[circle, draw=black](vi1a) at (3,.5){}; 
        \node[circle, draw=black](vi1b) at (3,-.5){}; 
        \node[circle, draw=black](vi2) at (2,.5){}; 
        \node[](vi2a) at(2.7,1.2){}; 
        \node[](e) at (.6,.2){\textcolor{red}{$e$}}; \node[circle,draw= black!50](v2) at(-.5,1){}; 
        \node[](v2a) at(-1.5,1.4){};
        \node[](v2b) at(-.7,1.9){};
        \node[] (T_i)at (1.7,-.4){$T_i$};
        \node[](v3) at(-.6,-.8){}; 
        \node[](v4) at(-1.1,.7){};
        \node[](k2) at (-.2,1.4){$ K_2$};
        \node[](k1) at (1,1.4){$K_1$};

        \draw[-,dashed, red](v)--(vi); 
        \draw[-](vi) to (vi1); 
        \draw[-](vi) to (vi2);
        \draw[-](vi1) to (vi1a);
        \draw[-](vi1) to (vi1b);
        \draw[-, black!30](v1) to (v1a); 
        \draw[-, black!30](v1) to (v1b); 
        \draw[-, black!30](v2) to (v2a);
        \draw[-, black!30](v2) to (v2b); 
        \draw[dashed,black!30](v) to (v2); 
        \draw[dashed,black!30](v) to (v1); 
        \draw[dashed,black!30](v) to (v3); 
        \draw[dashed,black!30](v) to (v4); 
        \draw[thick,dotted, black!80] (.4,1.8) to (.4,-.7); 
    \end{tikzpicture}
    \caption{Case I: $e$ is incident with $v$.}
    \label{fig:pfa}
    \end{subfigure}
    \hfill
    \begin{subfigure}[h!]{.45\textwidth}
    \centering 
    \begin{tikzpicture}
        \node[circle, draw=black!50](v)at (0,0){$v$};
        \node[circle, draw=black](vi)at (1,0){}; 
        \node[circle, draw=black](vi1)at (2,0){};
        \node[circle, draw=black](vi2)at (3,0){};
        \node[circle, draw=black](via)at (.7,.7){};
        \node[circle, draw=black](vib)at (1.2,.7){};
        
        \node[circle, draw=black](vi1a)at (2.4,.7){};
        \node[circle, draw=black](vi1b)at (2.4,-.7){};
        \node[circle, draw=black](vi1ab)at (3.4,.7){};
        \node[](e) at (1.7,.2){\textcolor{red}{$e$}}; 
        \node[](ti) at(1,-.5){$T_i$}; 
        
        \node[](k1)at(1.9,1.4){$K_1$}; 
        \node[](k2)at (1.1,1.4){$K_2$}; 
        
        \draw[dashed, black!50](v) to (vi); 
        \draw[dashed,black!50](v) to (-1,0);
        \draw[dashed,black!50](v) to (-.6,1);
        \draw[dashed,black!50](v) to (-.5,-.8);
        \draw[-,red] (vi) to (vi1);
        \draw[-](vi) to (via); 
        \draw[-](vi) to (vib);
        
        \draw[-](vi1) to (vi1a);
        \draw[-](vi1) to (vi1b);
        \draw[-](vi1a) to (vi1ab);
        \draw[-](vi1) to (vi2);
        \draw[thick,black!80, dotted](1.5,1.6) to (1.5, -1); 
        
    \end{tikzpicture}
    \caption{Case II: $e$ isn't incident with $v$.}
    \label{fig:pfb}
    \end{subfigure}
    
    \caption{The possible locations for a cut edge in a tree described in theorem \ref{thm:beans}.}
    \label{fig:pf}
\end{figure}
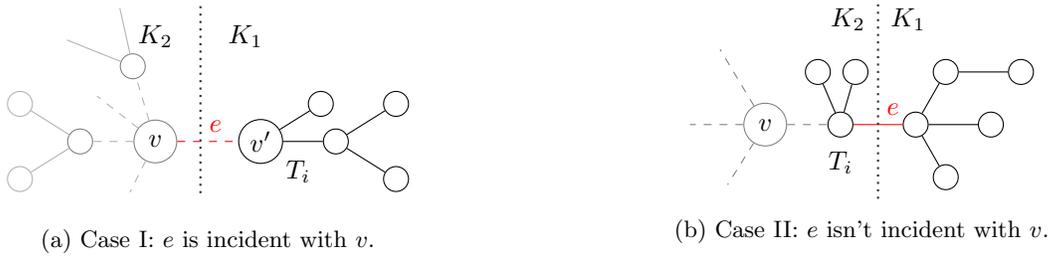
\begin{cor}
Let $T$ be a tree with total vertex weight $S$. Suppose that there exists a vertex $v$ in $T$ such that the forest $T-\{v\}$ has a component $T_i$ where $\frac{S}{2}-\epsilon\leq w(T_i) \leq \frac{S}{2}+\epsilon$. Then, $T$ is 2-splittable with a within an error of $\epsilon$ and the desired cut edge is the edge incident with $v$ and the vertex adjacent to $v$ in $T_i$.
\label{thm:beanscor}
\end{cor}

In figure \ref{fig:ex1}, removing vertex $v$ from the tree yields components with total vertex weights less than the desired $\frac{S}{2}-\epsilon$. Thus, we immediately know $T$ is not 2-splittable.

\begin{figure}[h!]
    \centering
    \begin{tikzpicture}
    
    \node[](tree) at (0,1.5){$T:$}; 
    
    \node[circle,draw=black, inner sep = .5mm] (a) at (0,0){
    \footnotesize{0.2}};
    
    \node[circle,draw=black, inner sep = .5mm] (b) at (1,0){
    \footnotesize{0.1}};
    
    \node[circle,draw=red, inner sep = .5mm] (c) at (2,0){
    \footnotesize{0.6}};
    \node[] (lab) at (1.7,.3){\textcolor{red}{$v$}}; 
    
    \node[circle,draw=black, inner sep = .5mm] (d) at (3,0){
    \footnotesize{0.3}};
    
    \node[circle,draw=black, inner sep = .5mm] (d1) at (4,0){
    \footnotesize{0.7}};
    
    \node[circle,draw=black, inner sep = .5mm] (d2) at (3.5,1){
    \footnotesize{0.4}};
    
    \node[circle,draw=black, inner sep = .5mm] (d3) at (3.5,-1){
    \footnotesize{0.2}};
    
    \node[circle,draw=black, inner sep = .5mm] (c_A) at (2,1){
    \footnotesize{0.1}};
    
    \node[circle,draw=black, inner sep = .5mm] (c_A1) at (1.5,2){
    \footnotesize{0.3}};
    
    \node[circle,draw=black, inner sep = .5mm] (c_A2) at (2.5,2){
    \footnotesize{0.5}};
    
    \node[circle,draw=black, inner sep = .5mm] (c1) at (1.5,-1){
    \footnotesize{0.6}};
    
    \node[circle,draw=black, inner sep = .5mm] (c2) at (2.5,-1){
    \footnotesize{0.4}};
    
    \node[circle,draw=black, inner sep = .5mm] (b1) at (1,1){
    \footnotesize{0.3}};
    
    \draw[-](a) to (b); 
    \draw[-](b) to (c); 
    \draw[-](c) to (d); 
    \draw[-](b) to (b1); 
    \draw[-](c) to (c_A);
    \draw[-](c) to (c1); 
    \draw[-](c) to (c2); 
    \draw[-](c_A) to (c_A1); 
    \draw[-](c_A) to (c_A2); 
    \draw[-](d) to (d1); 
    \draw[-](d) to (d2); 
    \draw[-](d) to (d3); 
    \end{tikzpicture}
    \hspace{1.5cm}
    \begin{tikzpicture}
    \node[](tree) at (-.5,1.5){$T-\{v\}:$};
    
     \node[circle,draw=black, inner sep = .5mm] (a) at (0,0){
    \footnotesize{0.2}};
    
    \node[circle,draw=black, inner sep = .5mm] (b) at (1,0){
    \footnotesize{0.1}};
    
    \node[circle,draw=black, inner sep = .5mm] (d) at (3,0){
    \footnotesize{0.3}};
    
    \node[circle,draw=black, inner sep = .5mm] (d1) at (4,0){
    \footnotesize{0.7}};
    
    \node[circle,draw=black, inner sep = .5mm] (d2) at (3.5,1){
    \footnotesize{0.4}};
    
    \node[circle,draw=black, inner sep = .5mm] (d3) at (3.5,-1){
    \footnotesize{0.2}};
    
    \node[circle,draw=black, inner sep = .5mm] (c_A) at (2,1){
    \footnotesize{0.1}};
    
    \node[circle,draw=black, inner sep = .5mm] (c_A1) at (1.5,2){
    \footnotesize{0.3}};
    
    \node[circle,draw=black, inner sep = .5mm] (c_A2) at (2.5,2){
    \footnotesize{0.5}};
    
    \node[circle,draw=black, inner sep = .5mm] (c1) at (1.5,-1){
    \footnotesize{0.6}};
    
    \node[circle,draw=black, inner sep = .5mm] (c2) at (2.5,-1){
    \footnotesize{0.4}};
    
    \node[circle,draw=black, inner sep = .5mm] (b1) at (1,1){
    \footnotesize{0.3}};
    
    \draw[-](a) to (b); 
    \draw[-](b) to (b1); 
    \draw[-](c_A) to (c_A1); 
    \draw[-](c_A) to (c_A2); 
    \draw[-](d) to (d1); 
    \draw[-](d) to (d2); 
    \draw[-](d) to (d3);
    \end{tikzpicture}
    \caption{The total vertex weight of $T$ is $S=4.7$. $T$ is not 2-splittable with an error of $\epsilon=.05$ because the total vertex weight of each component of $T-\{v\}$ is less than $2.35-.05=2.3$. This is easily verified by checking each edge of $T$.}
    \label{fig:ex1}
\end{figure}
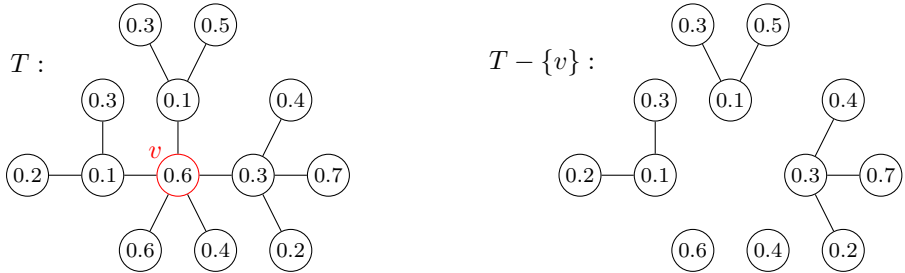
 
Although removing $v$ instantly showed $T$ to not be 2-splittable, it is important to note that if  the vertex adjacent to $v$ with weight 0.3 was removed, there would have been a component with total weight greater than $\frac{S}{2}+\epsilon$. This does not imply that $T$ is 2-splittable, as that would contradict our prior conclusion. If a tree $T$ is not 2-splittable, then every vertex $v_i$ of $T$ does \emph{not} necessarily yield components with every total weight less than $\frac{S}{2}-\epsilon$ when removed.  Thus, the converse of theorem \ref{thm:beans} is false. Figure \ref{fig:CE1} serves as an explicit counterexample to the converse. 
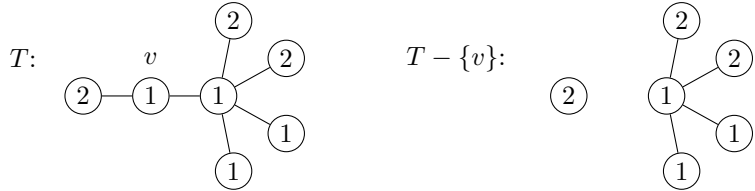
\begin{figure}[h!]
   \centering
    \begin{tikzpicture}
    \node[] (tree) at (.2,.5){$T$:};
    \node[circle,draw=black, inner sep = .7mm] (b) at (1,0){2};
    \node[circle,draw=black, inner sep = .7mm] (c) at (1.9,0){1};
    \node (k) at (1.9,.5){$v$};
    \node[circle,draw=black, inner sep = .7mm] (d) at (2.8,0){1};
    \node[circle,draw=black, inner sep = .7mm] (e) at (3,1){2};
    \node[circle,draw=black, inner sep = .7mm] (f) at (3.7,.5){2};
    \node[circle,draw=black, inner sep = .7mm] (g) at (3.7,-.5){1};
    \node[circle,draw=black, inner sep = .7mm] (h) at (3,-1){1};
    
    \draw[-](b) to (c); 
    \draw[-](c) to (d);
    \draw[-](d) to (e);
    \draw[-](d) to (f);
    \draw[-](d) to (g);
    \draw[-](d) to (h);
    \end{tikzpicture}
    \hspace{1cm}
    \begin{tikzpicture}
     \node[] (tree) at (0,.5){$T-\{v\}$:};
    \node[circle,draw=black, inner sep = .7mm] (b) at (1.5,0){2};
    \node[circle,draw=black, inner sep = .7mm] (d) at (2.8,0){1};
    \node[circle,draw=black, inner sep = .7mm] (e) at (3,1){2};
    \node[circle,draw=black, inner sep = .7mm] (f) at (3.7,.5){2};
    \node[circle,draw=black, inner sep = .7mm] (g) at (3.7,-.5){1};
    \node[circle,draw=black, inner sep = .7mm] (h) at (3,-1){1};

    \draw[-](d) to (e);
    \draw[-](d) to (f);
    \draw[-](d) to (g);
    \draw[-](d) to (h);
    \end{tikzpicture}
    \caption{ A counterexample to the converse of theorem \ref{thm:beans}. The tree $T$ is not 2-splitable with error $\epsilon=1$, but the graph $T-\{v\}$ contains a component with total weight $7>\frac{10}{2}-1=4$.}
    \label{fig:CE1}
    \end{figure}
    
In general, if a component in a forest $T-\{v\}$ has total vertex weight greater than $\frac{S}{2}+\epsilon$, then we must defer to the algorithm in section 2.1 to conclude if $T$ is 2-splittable. 

\subsection{An algorithm for finding a cut edge}

If we remove a vertex and there is component $T_{max}$ with total weight greater than  $\frac{S}{2}+\epsilon$, then we still need a way to find a cut edge or determine the tree is not 2-splittable. If the tree is 2-splittable, then any cut edge of $T$ must be in the component $T_{max}$. If the tree is not 2-splittable, then a vertex satisfying the conditions of theorem \ref{thm:beans} must also be in the component $T_{max}$. Therefore, it is only necessary to consider the component $T_{max}$. This argument leads to an algorithm for finding a cut edge of a given tree if it is 2-splittable.   

Let $T$ be a tree with total vertex weight $S$. We will find a cut edge for $T$ that makes the tree 2-splittable within a specified error $\epsilon$ or verify $T$ is not 2-splittable.
\begin{tcolorbox}
\textbf{Step 1.} Select a random vertex $v\in V(T)$. 

\textbf{Step 2.} Consider the forest $T-\{v\}$ with components $T_1, T_2,...,T_r$, where $r=deg(v)$. Calculate the total vertex weight for each component.
\begin{itemize}
\item Case 2a. If there exists a component $T_{max}$ with total vertex weight $W_1>\frac{S}{2}+\epsilon$, then proceed to step 3. 

\item Case 2b. If there exists a component $T_{max}$ with total weight $W_i$ where $W_i$ is within the desired range of $\frac{S}{2}\pm \epsilon$, then the desired cut vertex is the vertex incident with $v$ and some vertex in $V(T_{max})$ by corollary \ref{thm:beanscor}. We are done.  

\item Case 2c. If every component of $T-\{v\}$ is has total weight less than $\frac{S}{2}-\epsilon$, then by, theorem \ref{thm:beans}, $T$ is not 2-splittable and we are done. 
\end{itemize}

\textbf{Step 3.} Let the vertex $v_i\in V(T_{max})$ be adjacent to $v$ in $T$. Return to step 2 but substitute $T$ with $T_{max}$, $v$ with $v_i$, and consider the subgraph $T_{max}-\{v_i\}$ instead of $T-\{v\}$.
\end{tcolorbox}
\vspace{4mm}
If $T$ is 2-splittable, then a cut edge $e$ exists and therefore there is a vertex $v$ incident with $e$ that will produce a component with total weight within the desired range when removed. This vertex must be located in a component of the tree that has total weight greater than $\frac{S}{2}+\epsilon$ regardless of what vertex is initially removed. The algorithm always searches in this component, so it is guaranteed to find that vertex $v$. 

If $T$ is not 2-splittable and the initial vertex removed yields a component $T_{max}$ with total vertex weight greater than $\frac{S}{2}+\epsilon$, then  there is a vertex $v'$ in $T_{max}$ such that there is there is no component of $T_{max}-\{v'\}$ with total weight of at least $\frac{S}{2}-\epsilon$. Otherwise, $T$ would be 2-splittable. The algorithm is guaranteed to find this vertex since it looks in the components where such a vertex could exist. Thus, the algorithm will always terminate.

In figure \ref{fig:demo1}, the algorithm is demonstrated on the tree from figure \ref{fig:ex1} which was shown to be not 2-splittable. The algorithm takes two steps to show that $T$ is not 2-splittable, however, in figure \ref{fig:ex1} we chose a different vertex and it only took one step. Clearly, there are better and worse choices of vertices to start the algorithm from. We explore this in the next section.  

%In figure \ref{fig:ex1}, we proved tree was not 2-splittable by removing a vertex with weight $0.6$ in the approximate center of the tree. Removing this vertex immediately yielded a forest where every component had a total weight less than the desired error range. However, in figure \ref{fig:demo1}, the algorithm starts from a random vertex that does not immediately prove $T$ is not 2-splittable. 

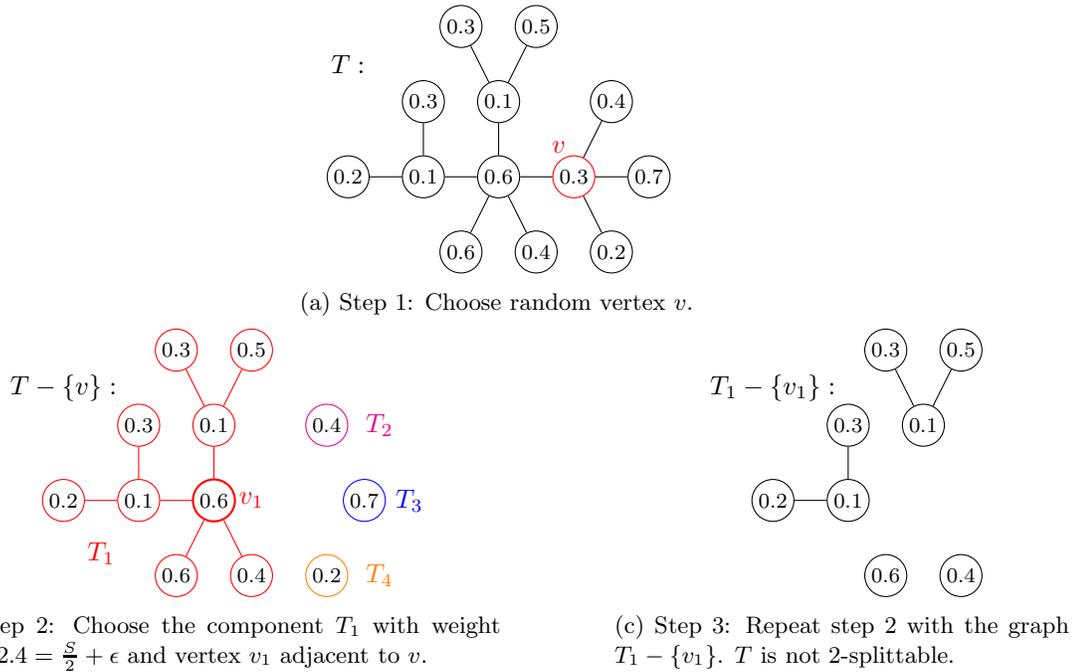
\begin{figure}[h!]
\centering
\begin{subfigure}[h]{.45\textwidth}
\centering

\begin{tikzpicture}
    \node[](tree) at (0,1.5){$T:$}; 
    \node[circle,draw=black, inner sep = .5mm] (a) at (0,0){
    \footnotesize{0.2}};
    
    \node[circle,draw=black, inner sep = .5mm] (b) at (1,0){
    \footnotesize{0.1}};
    
    \node[circle,draw=black, inner sep = .5mm] (c) at (2,0){
    \footnotesize{0.6}};
    
    \node[] at (2.80,.4){\textcolor{red}{$v$}};
    
    \node[circle,draw=red, inner sep = .5mm] (d) at (3,0){
    \footnotesize{0.3}};
    
    \node[circle,draw=black, inner sep = .5mm] (d1) at (4,0){
    \footnotesize{0.7}};
    
    \node[circle,draw=black, inner sep = .5mm] (d2) at (3.5,1){
    \footnotesize{0.4}};
    
    \node[circle,draw=black, inner sep = .5mm] (d3) at (3.5,-1){
    \footnotesize{0.2}};
    
    \node[circle,draw=black, inner sep = .5mm] (c_A) at (2,1){
    \footnotesize{0.1}};
    
    \node[circle,draw=black, inner sep = .5mm] (c_A1) at (1.5,2){
    \footnotesize{0.3}};
    
    \node[circle,draw=black, inner sep = .5mm] (c_A2) at (2.5,2){
    \footnotesize{0.5}};
    
    \node[circle,draw=black, inner sep = .5mm] (c1) at (1.5,-1){
    \footnotesize{0.6}};
    
    \node[circle,draw=black, inner sep = .5mm] (c2) at (2.5,-1){
    \footnotesize{0.4}};

    \node[circle,draw=black, inner sep = .5mm] (b1) at (1,1){
    \footnotesize{0.3}};
    
    \draw[-](a) to (b); 
    \draw[-](b) to (c); 
    \draw[-](c) to (d); 
    \draw[-](b) to (b1); 
    \draw[-](c) to (c_A);
    \draw[-](c) to (c1); 
    \draw[-](c) to (c2); 
    \draw[-](c_A) to (c_A1); 
    \draw[-](c_A) to (c_A2); 
    \draw[-](d) to (d1); 
    \draw[-](d) to (d2); 
    \draw[-](d) to (d3);
    \end{tikzpicture}
    \caption{Step 1: Choose random vertex $v$.}
    \end{subfigure}
    \vspace{1mm}
    
    \begin{subfigure}[h]{.5\textwidth}
    \centering
    \begin{tikzpicture}
    \node[](tree) at (0,1.5){$T-\{v\}:$}; 
    \node[circle,draw=red, inner sep = .5mm] (a) at (0,0){
    \footnotesize{0.2}};
    
    \node[circle,draw=red, inner sep = .5mm] (b) at (1,0){
    \footnotesize{0.1}};
    
    \node[circle,draw=red,thick, inner sep = .5mm] (c) at (2,0){
    \footnotesize{0.6}};
    
    \node[] at (2.5,0){\textcolor{red}{\textbf{$v_1$}}}; 
    
    \node[circle,draw=blue, inner sep = .5mm] (d1) at (4,0){
    \footnotesize{0.7}};
    \node[]at (4.6,0) {\textcolor{blue}{$T_3$}}; 
    
    \node[circle,draw=magenta, inner sep = .5mm] (d2) at (3.5,1){
    \footnotesize{0.4}};
    \node[]at (4.2,1) {\textcolor{magenta}{$T_2$}}; 
    
    \node[circle,draw=orange, inner sep = .5mm] (d3) at (3.5,-1){
    \footnotesize{0.2}};
    \node[] at (4.2,-1){\textcolor{orange}{$T_4$}}; 
    
    \node[circle,draw=red, inner sep = .5mm] (c_A) at (2,1){
    \footnotesize{0.1}};
    
    \node[circle,draw=red, inner sep = .5mm] (c_A1) at (1.5,2){
    \footnotesize{0.3}};
    
    \node[circle,draw=red, inner sep = .5mm] (c_A2) at (2.5,2){
    \footnotesize{0.5}};
    
    \node[circle,draw=red, inner sep = .5mm] (c1) at (1.5,-1){
    \footnotesize{0.6}};
    
    \node[circle,draw=red, inner sep = .5mm] (c2) at (2.5,-1){
    \footnotesize{0.4}};
    
    \node[circle,draw=red, inner sep = .5mm] (b1) at (1,1){
    \footnotesize{0.3}};
    \node[] at (.5,-.7){\textcolor{red}{$T_1$}};

    \draw[draw=red,-](a) to (b); 
    \draw[draw=red,-](b) to (c); 
    \draw[draw=red,-](b) to (b1); 
    \draw[draw=red,-](c) to (c_A);
    \draw[draw=red,-](c) to (c1); 
    \draw[draw=red,-](c) to (c2); 
    \draw[draw=red,-](c_A) to (c_A1); 
    \draw[draw=red,-](c_A) to (c_A2); 
    \end{tikzpicture}
    
    \caption{Step 2: Choose the component $T_1$ with weight $3.1>2.4 = \frac{S}{2}+\epsilon$ and vertex $v_1$ adjacent to $v$.}
    \end{subfigure}
    \hfill
    \begin{subfigure}[h]{.4\textwidth}
    \centering
    \begin{tikzpicture}
    \node[](tree) at (0,1.5){$T_1-\{v_1\}:$}; 
    \node[circle,draw=black, inner sep = .5mm] (a) at (0,0){
    \footnotesize{0.2}};
    
    \node[circle,draw=black, inner sep = .5mm] (b) at (1,0){
    \footnotesize{0.1}};
    
    \node[circle,draw=black, inner sep = .5mm] (c_A) at (2,1){
    \footnotesize{0.1}};
    
    \node[circle,draw=black, inner sep = .5mm] (c_A1) at (1.5,2){
    \footnotesize{0.3}};
    
    \node[circle,draw=black, inner sep = .5mm] (c_A2) at (2.5,2){
    \footnotesize{0.5}};
    
    \node[circle,draw=black, inner sep = .5mm] (c1) at (1.5,-1){
    \footnotesize{0.6}};
    
    \node[circle,draw=black, inner sep = .5mm] (c2) at (2.5,-1){
    \footnotesize{0.4}};
    
    \node[circle,draw=black, inner sep = .5mm] (b1) at (1,1){
    \footnotesize{0.3}};

    \draw[-](a) to (b); 
    \draw[-](b) to (b1); 
    \draw[-](c_A) to (c_A1); 
    \draw[-](c_A) to (c_A2);

    \end{tikzpicture}
    
    \caption{Step 3: Repeat step 2 with the graph $T_1-\{v_1\}$. $T$ is not 2-splittable.}
    \end{subfigure}
    
    \caption{Using the algorithm to show the tree from figure \ref{fig:ex1} with total weight 4.7 is not 2-splittable within an error $\epsilon=.05$.}
    \label{fig:demo1}
    \end{figure}
    
\subsection{An improved algorithm for finding a cut edge}

Selecting a vertex $v$ of a tree $T$ that maximizes the probability that no component of $T-\{v\}$ has total weight greater than $\frac{S}{2}+\epsilon$ would make the algorithm faster. In other words, a good vertex $v$ to start the algorithm from should minimize the average total weight of the components in $T-\{v\}$. 

Let $T$  be a tree with total vertex weight $S$. Choose vertex $v$ with weight $w(v)$ and degree $r$. Let $\overline{W}_v$ denote the average total vertex weight of the components of the forest $T-\{v\}$. We can express $\overline{W}_v$ as
$$\overline{W}_v=\frac{S-w(v)}{r}.$$

Starting the algorithm from a vertex $v$ that minimizes $\overline{W}_v$, will generally decrease the likelihood that there is a component of $T-\{v\}$ with a very large total vertex weight. This should then decrease the number of times the algorithm must iterate to find a cut edge. 

Since $\overline{W}_v$ is minimized when $v$ has high degree and large weight, a good vertex to start the algorithm from will have these properties. This fact motivates the following improved algorithm that starts from a vertex $v$ with a small $\overline{W}_v$ instead of a random one.

Let $T$ be a tree with total vertex weight $S$. We will find a cut edge for $T$ that makes the tree 2- splittable within a specified error $\epsilon$ or verify $T$ is not 2-splittable. 
\begin{tcolorbox}

\textbf{Step 1.} Let $\Delta(T)$ denote the maximum degree of any vertex in $T$ and $S=\{d_1,d_2,...,d_n\}$ be the set of vertices with degree $\Delta(T)$. Choose the vertex $v$ to be the vertex with maximum weight in $S$.

\textbf{Step 2.} Proceed with steps 2 and 3 stated in section 2.1.
\end{tcolorbox}
\vspace{4mm}

Using this improved algorithm on the tree in figure \ref{fig:ex1}, would start from the vertex $v$ shown in figure \ref{fig:ex1} instead of the random vertex chosen in figure \ref{fig:demo1}.

\section{Application to Unweighted Trees}

The algorithm can also be applied  to find two subtrees of equal order in an unweighted tree. It follows as a corollary of theorem \ref{thm:beans}.
\begin{cor}
Let $T$ be an unweighted tree of  even order $n$. Suppose that there exists a vertex in $T$ such that every component of the forest $T-\{v\}$ has order less than $\frac{n}{2}$. Then there is no edge in $T$ that breaks the tree into two components of equal order.
\end{cor}
\begin{proof}
Begin by assigning a weight of 1 to every vertex in $T$. The total weight is then equal to the order $n$ of the tree. Then, by theorem \ref{thm:beans} with $\epsilon=0$, the corollary follows.
\end{proof}

\printbibliography 

@online{gerrychain,
    author = "MGGG",
    title = "GerryChain",
    url  = "https://github.com/mggg/GerryChain"
}

@misc{deford2019recombination,
      title={Recombination: A family of Markov chains for redistricting}, 
      author={Daryl DeFord and Moon Duchin and Justin Solomon},
      year={2019},
      eprint={1911.05725},
      archivePrefix={arXiv},
      primaryClass={cs.CY}
      
}
Corinne Mulvey, 
\href{mailto:corinne.mulvey@gmail.com}{corinne.mulvey@gmail.com}
\end{document}